\documentclass[12pt]{amsart}
\usepackage{amscd}
\usepackage{lineno}
\usepackage{pstcol,pst-plot,pst-3d}
\def\NZQ{\Bbb}               % the font for N,Z,Q,R,C
\def\NN{{\NZQ N}}

\def\ZZ{{\NZQ Z}}

\def\FFF{{\NZQ F}}

\def\frk{\frak}               % font for "Fraktur"

\def\mm{{\frk m}}

\def\opn#1#2{\def#1{\operatorname{#2}}} % to make operators
\opn\chara{char} \opn\length{\ell} \opn\pd{pd} \opn\rk{rk}
\opn\projdim{proj\,dim} \opn\injdim{inj\,dim} \opn\rank{rank}
\opn\depth{depth} \opn\codepth{codepth} \opn\grade{grade} \opn\height{height}
\opn\embdim{emb\,dim} \opn\codim{codim}

\opn\Tr{Tr} \opn\bigrank{big\,rank}
\opn\superheight{superheight}\opn\lcm{lcm}
\opn\trdeg{tr\,deg}%
\opn\reg{reg} \opn\lreg{lreg} \opn\skel{skel}
\opn\div{div} \opn\Div{Div} \opn\cl{cl} \opn\Cl{Cl}
\opn\Spec{Spec} \opn\Supp{Supp} \opn\supp{supp} \opn\Sing{Sing}
\opn\Ass{Ass}
\opn\Ann{Ann} \opn\Rad{Rad} \opn\Soc{Soc}    \opn\Fitt{Fitt}
\opn\Sym{Sym} \opn\Ker{Ker} \opn\Coker{Coker} \opn\Im{Im}
\opn\Hom{Hom} \opn\Tor{Tor} \opn\Ext{Ext} \opn\End{End}
\opn\Aut{Aut} \opn\id{id} \opn\ini{in}

\opn\nat{nat}\opn\it{it}
\opn\pff{proof}%   \pf exists already
\opn\Pf{proof} \opn\GL{GL} \opn\SL{SL} \opn\mod{mod} \opn\ord{ord} \opn\H{H}
\opn\Hilb{Hilb}
\opn\aff{aff} \opn\con{conv} \opn\relint{relint} \opn\st{st}
\opn\lk{lk} \opn\cn{cn} \opn\core{core} \opn\vol{vol}
\opn\link{link} \opn\star{star} \opn\skel{skel}\opn\Kd{Krull-dim}
\opn\gr{gr}

\def\pot#1#2{#1[\kern-0.28ex[#2]\kern-0.28ex]}

\opn\dirlim{\underrightarrow{\lim}}
\opn\inivlim{\underleftarrow{\lim}}
\let\tensor=\otimes
\let\iso=\cong

\let\Dirsum=\bigoplus

\let\to=\rightarrow

\def\Implies{\ifmmode\Longrightarrow \else
     \unskip${}\Longrightarrow{}$\ignorespaces\fi}
\def\implies{\ifmmode\Rightarrow \else
     \unskip${}\Rightarrow{}$\ignorespaces\fi}
\def\iff{\ifmmode\Longleftrightarrow \else
     \unskip${}\Longleftrightarrow{}$\ignorespaces\fi}

\let\:=\colon
\newtheorem{Theorem}{Theorem}[section]
\newtheorem{Lemma}[Theorem]{Lemma}
\newtheorem{Corollary}[Theorem]{Corollary}
\newtheorem{Proposition}[Theorem]{Proposition}

\let\epsilon\varepsilon
\let\phi=\varphi
\let\kappa=\varkappa
\def\qed{\ifhmode\textqed\fi
   \ifmmode\ifinner\quad\qedsymbol\else\dispqed\fi\fi}
\def\textqed{\unskip\nobreak\penalty50
    \hskip2em\hbox{}\nobreak\hfil\qedsymbol
    \parfillskip=0pt \finalhyphendemerits=0}
\def\dispqed{\rlap{\qquad\qedsymbol}}

\opn\ini{in} \opn\inim{inm} \opn\rate{rate}

\opn\codim{codim}

\begin{document}
%\linenumbers

\title{Bounds for  the regularity of local cohomology of bigraded modules}
\author{J\"urgen Herzog and Ahad Rahimi}
%\date{}
%\maketitle

\subjclass[2000]{  13D45, 13D02, 16W50.  This paper was written during the visit of the second author at Universit\"at Duisburg-
Essen, Campus Essen. He is grateful for its hospitality.    The second author was in
part supported by a grant from IPM (No. 91130029)}

%\subjclass{13D45, 13D40, 13D02, 13P10}
%%\thanks{}
\address{J\"urgen Herzog, Fachbereich Mathematik, Universit\"at Duisburg-Essen, Campus Essen, 45117
Essen, Germany} \email{juergen.herzog@uni-essen.de}
\address{ Ahad Rahimi, Department of Mathematics, Razi University, Kermanshah,
 Iran and
 School of Mathematics, Institute for Research in Fundamental Sciences
(IPM), P. O. Box: 19395-5746, Tehran, Iran.
}\email{ahad.rahimi@razi.ac.ir}

\begin{abstract}
Let $M$ be a finitely generated bigraded module over the standard bigraded polynomial ring $S=K[x_1,\ldots,x_m, y_1,\ldots,y_n]$, and let $Q=(y_1,\ldots,y_n)$. The local cohomology modules $H^k_Q(M)$ are naturally bigraded, and the components $H^k_Q(M)_j=\Dirsum_iH^k_Q(M)_{(i,j)}$ are finitely generated graded $K[x_1,\ldots,x_m]$-modules. In this paper we study the regularity of $H^k_Q(M)_j$, and show in several cases that $\reg H^k_Q(M)_j$ is linearly bounded as a function of $j$.
\end{abstract}

\maketitle

\section*{Introduction}

In this paper  we  study the regularity of local cohomology of bigraded modules.  Let $K$ be a field and $S=K[x_1,\ldots,x_m,y_1,\ldots,y_n]$ be the polynomial in the variables $x_1,\ldots,x_m,y_1,\ldots,y_n$. We consider $S$ to be a standard bigraded $K$-algebra with $\deg x_i=(1,0)$ and $\deg y_j=(0,1)$ for all $i$ and $j$. Let $I\subset S$ be a bigraded ideal. Then $R=S/I$ is again a standard bigraded $K$-algebra. Let $M$ be a finitely generated bigraded $R$-module. We consider the local cohomology modules  $H^k_Q(M)$ with respect to $Q=(y_1,\ldots,y_n)$.  This module has a natural bigraded $S$-module structure. For all integers $j$ we set
\[
H^k_Q(M)_j=\Dirsum_{i}H^k_Q(M)_{(i,j)}.
\]
Notice that $H^k_Q(M)_j$ is a finitely generated graded $S_0$-module, where $S_0$ is the polynomial ring $K[x_1,\ldots,x_m]$.

The main purpose of this paper is to study the regularity of the $S_0$-modules  $H^k_Q(M)_j$ as a function of $j$. In all known cases,  $\reg H^k_Q(M)_j$ is bounded above (or equal to) a linear function $aj+b$ for suitable integers $a$ and $b$ with $a\leq 0$, see \cite{AR1}. Various cohomological conditions on $M$ are known that guarantee that for suitable $k$ the regularity of  $H^k_Q(M)_j$ as a function of $j$ is  actually bounded,  see the papers \cite{BRR}, \cite{AR2} and \cite{AR3}. For example, in \cite[Corollary 2.8]{AR3} it is shown that if $M$ is sequentially Cohen--Macaulay with respect to $Q$, then  there exists a number  $c$ such that  $|\reg H^k_Q(M)_j|\leq c$ for all $k$ and $j$. In general however one can expect only linear bounds. As shown in \cite[Theorem 5.3 and Corollary 5.4]{AR1} the  regularity of the local cohomology modules $H_Q^k(R)_j$ is linearly bounded, if $R=S/(f)$ is a hypersurface ring for which the content ideal $c(f)\subset S_0$ is $\mm_0$-primary, where $\mm_0=(x_1,\ldots,x_m)$ is the graded maximal ideal of $S_0$. Here, for a  bihomogeneous polynomial $f=\sum_{|\beta|=b} f_\beta y^\beta$, the ideal  $c(f)$ is the ideal in $S_0$ defined by the polynomials $f_\beta$. If, more generally,  $R=S/I$, where $I$ is a bihomogeneous ideal, one defines the content ideal $c(I)$ of $I$ to be the ideal generated in $S_0$ by the polynomials $c(f)$  with $f$ in $I$. As the main result of  Section~\ref{anni} we show in Theorem~\ref{Anihilate} that there exists a linear function $\ell\: \ZZ_-\to \ZZ_+$ such that $c(I)^{\ell(j)}H_Q^n(R)_j=0$ for all $j$.  This result is then used to generalize the above quoted theorem about hypersurface rings and to obtain that $\reg H_Q^n(R)_j$ is linearly bounded, provided $c(I)$ is  $\mm_0$-primary, see Corollary~\ref{primary}. Our Theorem~\ref{top} generalizes \cite[Corollary 5.4]{AR1} in a different direction.  It is shown in this theorem that the regularity of  $H_Q^n(R)_j$  is linearly bounded provided that $\dim S_0/c(I)\leq 1$. It is an open question whether the condition on the dimension of $S_0/c(I)$ can be dropped in this statement.

Some more explicit results  concerning the regularity of  $H_Q^n(R)_j$, in the case that $R=S/(f)$ is a hypersurface ring,  are obtained in Section~\ref{anni}. In Proposition~\ref{regular} we show that $\reg H_Q^n(R)_j$ is a linear function if $f=\sum_{i=1}^nf_iy_i$ where $f_1,\ldots, f_n$ is a regular sequence, and in Corollary~\ref{twosummands} it is shown that $\reg H_Q^n(R)_j$ is a linear function if  $f=f_1y_1+f_2y_2$ and $\deg \gcd(f_1,f_2)<\deg f_i$.

Unfortunately, the methods used  in Section 1 to bound the regularity can be used only for the top local cohomology $H^n_Q(R)$ and only when the content ideal of the defining ideal of $R$ is $\mm_0$-primary, or in the hypersurface case an ideal of height $m-1$. The situation is much better when we consider local cohomology  of multigraded $S$-modules. Indeed, in Section~2 of the paper it is shown in Theorem~\ref{againinessen} that if  $M$ is a finitely generated  $\ZZ^m\times \ZZ^n$-graded $S$-module. Then there exists an integer $c$, which only depends on the $x$-shifts of the bigraded resolution of $M$,  such that
$|\reg H^k_Q(M)_j|\leq c$  for all $k$ and all $j$. As mentioned above such a bound also exists when $M$ is only bigraded, but sequentially Cohen-Macaulay with respect to $Q$. The more it is surprising that in the multigraded case, no other cohomological condition on $M$ is required to obtain such a global bound. The proof of Theorem~\ref{againinessen} uses essentially a result of  Bruns and the first author  \cite[Theorem 3.1]{BH1} which says that the multigraded shifts in the resolution of a multigraded $S$-module $M$ can be bounded in terms of the multigraded degrees of the generators of the first relation module of $M$.

\section{On the annihilation of the graded components of top local cohomology with applications to regularity bounds}
\label{anni}

Let $I=(f_1,\ldots,f_r)$ where the $f_i$ are bihomogeneous polynomials. Then the {\em content ideal} $c(I)$ of $I$ is defined to be the ideal $c(f_1)+\cdots  +c(f_r) \subset S_0$ where for  bihomogeneous polynomial $f=\sum_{|\beta|=b} f_\beta y^\beta$ the ideal  $c(f)$ is the ideal in $S_0$ defined by the polynomials $f_\beta$. Obviously, the definition of $c(I)$ does not depend on the chosen set of generators of $I$.  In this section we show that a certain power  of the content ideal of $I$  annihilates $H^n_Q(R)_j$, and use this fact to bound the regularity of  $H^n_Q(R)_j$ in some cases.

%\begin{Lemma}
%\label{content}
%Let $f_1, \ldots, f_r$ be the bihomogeneous polynomials in $S$ with $\deg (f_i)=(a_i, b_i)$ for $i=1, \ldots, r$. We set $I=(f_1, \ldots, f_r)$ and $R=S/I$. Then
%\[
%\dim _{S_0}H^n_Q(R)_j\leq \dim S_0/c(I),
%\]
%where $c(I)$ is the content ideal of $I$. In particular, if $c(I)$ is $\mm_0$-primary, then $H^n_Q(R)_j$ is of finite length for all $j$.
%\end{Lemma}
%\begin{proof}
 %Let
% \[
% \cdots \longrightarrow \Dirsum_{i=1}^rS(-a_{i}, -b_{i}) \overset {U} \longrightarrow S \longrightarrow R \longrightarrow 0,
 %\]
%be a $S$-graded free resolution of $R$ where $U=(f_1, \ldots, f_r)$. Thus $H_{Q}^n(R)_j$ has the following $S_0$-presentation
%\[
% \cdots \longrightarrow \Dirsum_{i=1}^r S_0^{n_i}(-a_i) \overset {U_j} \longrightarrow S_0^{n_0} \longrightarrow H^n_Q(R)_j \longrightarrow 0,
% \]
% where $n_0=\binom{-j-1}{n-1}$ and $n_i=\binom{-j+b_i-1}{n-1}$ for $i=1, \ldots, r$.  In view of Formula (5) in \cite{AR1} we have $\dim H^n_Q(R)_j=\dim S_0/I_{n_0}(U_j)$ where $I_{n_0}(U_j)$ is the ideal generated by the $n_0$-minors of matrix $U_j$. By \cite[Lemma 1.4]{KS} we have $c(I)\subseteq \sqrt{I_{n_0}(U_j)}$. Thus $\dim _{S_0}H^n_Q(R)_j\leq \dim S_0/c(I)$, as desired.
%\end{proof}

\begin{Theorem}
\label{Anihilate}
Let $R=S/I$ where $I$ is a bigraded ideal  in $S$. Then there exists a linear function $\ell\: \ZZ_-\to \ZZ_+$ such that $c(I)^{\ell(j)}H^n_Q(R)_j=0$ for all $j\leq 0$.
\end{Theorem}
\begin{proof} Let $f_1,\ldots,f_r$ be a minimal set of bihomogeneous polynomials generating the ideal $I$, and assume that $f_i$ is homogeneous of bidegree $(-a_{i}, -b_{i})$. Then $R$ has the following free presentation
  \[
 \cdots \longrightarrow \Dirsum_{i=1}^rS(-a_{i}, -b_{i}) \longrightarrow S \longrightarrow R \longrightarrow 0.
 \]
By \cite[Theorem 1.1]{AR1}, the $S_0$-module $H_{Q}^n(R)_j$ has then the following free $S_0$-presentation
\[
 \cdots \longrightarrow \Dirsum_{i=1}^r H^n_Q(S)(-a_{i}, -b_{i})_j \overset {\phi} \longrightarrow  H^n_Q(S)_j  \longrightarrow H^n_Q(R)_j \longrightarrow 0.
 \]
The map  $\phi$ can be described as follows: We may write $f_i=\sum_{|\beta|=b_i}f_{i, \beta}y^\beta$ with $\deg f_{i, \beta}=a_i$ for $i=1, \ldots, r$ and set  $R=K[y_1, \dots, y_n]$.  By Formula (1) in \cite{AR1} we have
\[
G_0: = H^n_Q(S)_j=\Dirsum_{\left|c\right|=-n-j}S_0z^c,
\]
 where $z\in \Hom_K(R_{-n-j}, K)$ is the $K$-linear map with
 \[
\ z^a(y^b)=\left\{
\begin{array}{ll}
z^{a-b}, & \text{if $b\leq a$,}\\
 0, & \text{if  $b\not \leq a$.}
\end{array}
\right.
\]
We set $G_1=\Dirsum_{i=1}^rF_i$ where
\[
  F_i= H^n_Q(S)(-a_{i}, -b_{i})_j= \Dirsum_{\left|c\right|=-n-j+b_i}S_0(-a_{i})z^c.
  \]
  For $G_0$ the basis consists the elements  $z^c$ with  $\left|c\right|=-n-j$ and $G_1$ has a basis consisting of the elements $e_iz^c$ with $\deg e_i=a_i$ and $\left|c\right|=-n-j+b_i$.
We have $\phi(e_iz^c)=f_{i, \beta}z^{c-\beta}$ with $|\beta|=b_i$ if $\beta\leq c$,  and otherwise $0$.
 We set $\phi(F_i)=U_i$ for $i=1, \ldots, r$. Then  $\Im \phi=\sum_{i=1}^rU_{i}$. We set   $T_i=K[\{x_{i, \beta} \}_{\beta \in \NN^n, |\beta|=b_i}]$, $P_i=T_i[y_1, \ldots, y_n]$  and $g_i=\sum_{|\beta|=b_i}x_{i, \beta}y^\beta$ for $i=1, \ldots, r$.
  By \cite[Proposition 5.2]{AR1} there exists a linear function $l_i(j)$ such that
 \[
\big[ H^n_Q\big(P_i/(g_i)\big)_j\big ]_{l_i(j)} =0 \quad\text{for}\quad  i=1, \ldots, r.
 \]
 Since the graded $T_i$-module $H^n_Q(P_i/(g_i))_j$ is generated in degree $0$,  it follows that $\mm_i^{l_i(j)}H^n_Q(P_i/(g_i))_j=0$ where  $\mm_i=(\{x_{i,\beta} \}_{\beta \in \NN^n, |\beta|=b_i})$ for $i=1, \ldots, r$.
    Replacing $x_{i, \beta}$ by $f_{i, \beta}$ we obtain
      \begin{eqnarray}
      \label{formula1}
      c(f_i)^{l_i(j)}H^n_Q(S/(f_i))_j=0 \quad\text{for}\quad  i=1, \ldots, r.
      \end{eqnarray}
Indeed,  in order  to prove (\ref{formula1}),  we consider the map $\pi: T_i\to S_0$ where $x_{i, \beta} \longmapsto f_{i, \beta}$. It follows from the free presentations of $H^n_Q(P_i/(g_i))_j$ and $H^n_Q(S/(f_i))_j$ that
\begin{eqnarray}
\label{pi}
H^n_Q(S/(f_i))_j \iso H^n_Q(P_i/(g_i))_j\tensor_{T_i}S_0.
\end{eqnarray}
Notice that if  $m\in H^n_Q(P_i/(g_i))_j$ and $h\in T_i$, then we have
\[
(hm)\tensor 1=m\tensor \pi(h)=\pi(h)(m\tensor 1).
 \]
 In particular,  if $h\in \mm_i^{l_i(j)}$, then   $\pi(h)(m\tensor 1)=(hm)\tensor 1=0$, and hence by (\ref{pi}) we have  $\pi(h) H^n_Q(S/(f_i))_j =0$ for all $h\in \mm_i^{l_i(j)}$. Therefore $c(f_i)^{l_i(j)}H^n_Q(S/(f_i))_j=0$, because $\pi(\mm_i^{l_i(j)})=c(f_i)^{l_i(j)}$.

    Since  the functor $H^n_Q(-)$ is right exact, the canonical epimorphism $S/(f_i)\to S/I=R$ induces and epimorphism  $H^n_Q(S/(f_i))_j \to H^n_Q(R)_j$ for all $j$. It follows that
      \[
      c(f_i)^{l_i(j)}H^n_Q(R)_j=0 \quad\text{for}\quad  i=1, \ldots, r.
      \]
  We set $\ell(j)=\sum_{i=1}^r{l_i(j)}$. Then $\ell(j)$ is again a linear function of $j$, and
 \[
 c(I)^{\ell(j)}H^n_Q(R)_j=\Big(\sum_{i=1}^rc(f_i)\Big)^{\ell(j)}H^n_Q(R)_j=0,
  \]
  as desired.
\end{proof}

\medskip
Let $N$ be a $\ZZ$-graded $S_0$-module, and
\[
0 \rightarrow F_k \rightarrow \cdots \rightarrow F_1 \rightarrow F_0 \rightarrow N \rightarrow 0,
\]
be  the minimal graded free $S_0$-resolution of $N$ with  $F_i=\Dirsum_ {j=1}^{t_i} S_0(-a_{ij})$ for $i=1,\ldots,k$.
Then the {\em Castelnuovo-Mumford regularity} $\reg N$ of $N$ is defined to be the integer
\[
\reg N= \max_{i,j}\{a_{ij}-i\}.
\]
\begin{Corollary}
\label{primary}
Assume in addition to Theorem \ref{Anihilate} that $c(I)$ is an $\mm_0$-primary ideal where  $\mm_0$ is the graded maximal ideal of $S_0$. Then there exists a linear function $\ell: \ZZ_-\to \ZZ_+$ such that
\[
0\leq \reg H^n_Q(R)_j \leq \ell(j)
\]
for all $j$.
\end{Corollary}
\begin{proof}
The lower bound for the regularity follows from the fact that $H^n_Q(R)_j $ is generated in degree $0$.
Since $c(I)$ is an $\mm_0$-primary, it follows that $\mm_0^k\subseteq c(I)$ for some $k$ and hence by Theorem \ref{Anihilate}, there exists a linear function $\ell'$ such that
\[
\mm_0^{k{\ell'(j)}}H^n_Q(R)_j=c(I)^{\ell'(j)}H^n_Q(R)_j=0.
\]
Therefore $H^n_Q(R)_j$ is of finite length for all $j$ and $\reg H^n_Q(R)_j \leq k\ell'(j)-1=\ell(j)$.

\end{proof}

\begin{Corollary}
\label{dim1}
With the assumptions and the notation of Theorem \ref{Anihilate} we have
\[
\dim _{S_0}H^n_Q(R)_j= \dim S_0/c(I) \quad\text{for all}\quad j.
\]
\end{Corollary}
\begin{proof}
As we have already seen, $H_{Q}^n(R)_j$ has the following $S_0$-presentation
\[
 \cdots \longrightarrow \Dirsum_{i=1}^r S_0^{n_i}(-a_i) \overset {\phi_j} \longrightarrow S_0^{n_0} \longrightarrow H^n_Q(R)_j \longrightarrow 0,
 \]
 where $n_0=\binom{-j-1}{n-1}$ and $n_i=\binom{-j+b_i-1}{n-1}$ for $i=1, \ldots, r$. Let $U_j$ be the matrix describing $\phi_j$ with respect to the canonical bases. Notice that $I_{n_0}(U_j)\subseteq c(I)$  where $I_{n_0}(U_j)$ is the ideal generated by the $n_0$-minors of matrix $U_j$. Thus
 \[
 \dim S_0/c(I)\leq \dim S_0/I_{n_0}(U_j)=\dim H^n_Q(R)_j .
 \]
Here the second equality follows from Formula (5) in \cite{AR1}. On the other hand,  Theorem \ref{Anihilate} implies that  $\dim H^n_Q(R)_j \leq \dim S_0/c(I)$. Therefore the desired equality follows.
\end{proof}

The following known fact is needed for the proof of the next corollary. For the convenience of the reader we include its proof.
\begin{Lemma}
\label{generic}
Let $M$ be a graded $S_0$-module with $\dim M>0$ and $|K|=\infty$. Suppose $f$ be a linear form such that $0:_Mf$ has finite length. Then
\[
\dim M/fM=\dim M-1.
\]
\end{Lemma}
\begin{proof}
We denote by $\H_M(t)= \sum_{i\in \ZZ}\dim_K M_it^i$ the Hilbert-series of $M$. Consider the exact sequence $0\to 0:_Mf \to M(-1)  \overset{f}\to M \to M/fM \to 0$. Since $(0:_Mf )$ has finite length, it follows that $\H_{(0:_Mf )}(t)=Q_0(t)$, where $Q_0(t)$ is a polynomial in $t$ with $Q_0(1)\neq 0.$  We may also write  $\H_M(t)=Q(t)/(1-t)^d$ where $d=\dim M$ with $Q(1)\neq 0$. Hence
\begin{eqnarray*}
\H_{M/fM}(t) &= & \H_M(t)-t\H_M(t)+\H_{0:_Mf}(t) \\&=&   P(t)/(1-t)^{d-1}
\end{eqnarray*}
where  $P(t)=Q(t)+(1-t)^{d-1}Q_0(t)$  with $P(1)\neq 0$. Thus the desired equality follows from \cite[Corollary 4.1.8]{BH}.
\end{proof}

As another application of Theorem \ref{Anihilate} we have

\begin{Theorem}
\label{top}
Consider the hypersurface ring $R=S/fS$ where $f$ is a bihomogeneous polynomial in $S$. Suppose that $\dim S_0/c(f)\leq 1$. Then there exists a linear function $\ell: \ZZ_-\to \ZZ_+$ such that
\[
0\leq \reg H_Q^k(R)_j \leq \ell(j)
\]
for  all $k$ and  $j$.
\end{Theorem}

\begin{proof}
 Note that $R$ has only two non-vanishing local cohomology  modules  $H_Q^k(R)$, namely for  $k=n-1$ and $k= n$. A similar argument as that used in  the proof of  \cite[Proposition 5.1]{AR1}(b) shows that $H^{n-1}_Q(R)_j$ is  linearly bounded, provided $H^{n}_Q(R)_j$ is  linearly bounded. Thus it suffices to consider the case $k=n$. The desired result follows from Corollary \ref{primary} in the case that $\dim S_0/c(f)=0$. Now let us assume that $\dim S_0/c(f)=1$.
We may assume that $K$ is infinite. Otherwise, we apply a suitable base field extension. Thus we can choose a linear form $g$ such that $0:_{H^n_Q(R)_j}g$ and $0:_{S_0/c(f)}g$ are finite length modules. After a change of coordinate we may assume that $g=x_1.$ For any $S_0$-module $M$ we set $\overline{M}=M/x_1M.$ We first observe  that $\overline{H^n_Q(R)_j}\iso H^n_Q(\overline{R})_j$. Indeed, since the functor $H^n_Q(-)$ is right exact,  the exact sequence $ R  \overset {x_1} \longrightarrow R \rightarrow \overline{R} \to 0$  induces the exact sequence  $ H^n_Q(R)_j  \overset {x_1} \longrightarrow H^n_Q(R)_j \rightarrow H^n_Q(\overline{R})_j \to 0$,  which yields the desired isomorphism.

Next observe that  $\overline{S_0/c(f)}=\overline{S_0}/c(\overline{f})$  where $\overline{f}$ is the image of $f$ under the canonical epimorphism $S_0\to \overline{S_0}$. Identifying $\overline{S_0}$ with $K[x_2,\ldots,x_n]$
the map $f\mapsto \overline{f}$ is obtained by substituting $x_1$ by $0$. Since $\dim S_0/c(f)=1$ and since $0:_{S_0/c(f)}x_1$ has finite length, it follows that
$\dim \overline{S_0}/c(\overline{f})=\dim \overline{S_0/c(f)}=\dim S_0/c(f) -1=0$, see Lemma~\ref{generic}.  By \cite[Proposition 20. 20]{Ei} we have
\begin{eqnarray}
\label{E}
\reg H^n_Q(R)_j=\max\{ \reg\big( 0:_{H^n_Q(R)_j}x_1\big), \reg\overline{H^n_Q(R)_j} \}.
\end{eqnarray}
Moreover, by Corollary  \ref{dim1},  we have  $\dim H^n_Q(\overline{R})_j= \dim \overline{S_0}/c(\overline{f})=0$. Hence, it follows from \cite[Theorem 5.3]{AR1} that $\reg\overline{H^n_Q(R)_j} =\reg H^n_Q(\overline{R})_j$  is linearly bounded as a function of $j$. Now we show that $\reg\big( 0:_{H^n_Q(R)_j}x_1\big)$ is linearly bounded, as well. Then by (\ref{E}) the desired conclusion follows. The exact sequence
$0\to  R  \overset {x_1} \longrightarrow R \rightarrow \overline{R} \to 0$  induces the exact sequence
%\begin{eqnarray}
%\label{E1}
\[
 H^{n-1}_Q(R)_j \longrightarrow H^{n-1}_Q(\overline{R})_j \rightarrow 0:_{H^n_Q(R)_j}x_1 \to 0.
 \]
% \end{eqnarray}
In particular, we conclude that the highest degree of a  generator of $0:_{H^n_Q(R)_j}x_1$ is less than or equal to the highest degree of a generator of $H^{n-1}_Q(\overline{R})_j$.
As  $\reg H^n_Q(\overline{R})_j$ is linearly bounded, it follows  that there exists a linear function $\ell'$ such that $\reg H^{n-1}_Q(\overline{R})_j\leq \ell'(j)$, see the proof of \cite[Proposition 5.1]{AR1}(b). Thus this linear function $\ell'$  also bounds the highest degree of a  generator of $H^{n-1}_Q(\overline{R})_j$, and hence  the highest degree of a  generator of  $0:_{H^n_Q(R)_j}x_1$.

By Theorem \ref{Anihilate} there exists a linear function $\ell''$ such that
\[
c(f)^{\ell''(j)} (0:_{H^n_Q(R)_j}x_1) \subseteq c(f)^{\ell''(j)}H^n_Q(R)_j=0.
\]
Since $\dim \overline{S_0}/c(\overline{f})=0$,  it follows that  there exists an integer $k$ such that $\overline{\mm_0}^k\subseteq c(\overline{f})$. Hence we obtain
\begin{eqnarray*}
\mm_0^{k\ell''(j)}(0:_{H^n_Q(R)_j}x_1)&=&\overline{\mm}_0^{k\ell''(j)}(0:_{H^n_Q(R)_j}x_1) \\ &\subseteq & c(\overline{f})^{\ell''(j)}(0:_{H^n_Q(R)_j}x_1)\\
&=&c(f)^{\ell''(j)}(0:_{H^n_Q(R)_j}x_1)=0.
\end{eqnarray*}
We conclude that $( 0:_{H^n_Q(R)_j}x_1)_{i}=0$ for $i\geq \ell'(j)+k\ell''(j)$ and therefore,
\[
\reg ( 0:_{H^n_Q(R)_j}x_1) \leq \ell'(j)+k\ell''(j).
\]
\end{proof}
%\begin{Corollary}
%let $f\in S$ be a bihomogeneous polynomial and consider the hypersurface ring $R=S/fS$. Then the regularity of $H^i_Q(R)_j$ for $i=n, n-1$ are linearly bounded.
%\end{Corollary}
%\begin{proof}
%The assertion follows from Corollary \ref{top} for $i=n$. For $i=n-1$, we use the same argument as in the proof of \cite[Proposition 5.1]{AR1}(b).
%\end{proof}
For a hypersurface ring of bidegree $(d, 1)$ we have the following more precise result.

\begin{Proposition}
\label{regular}
Let $R=S/fS$ where $f=\sum_{i=1}^nf_iy_i$ with $\deg f_i=d$ for all $i$. If $f_1, \ldots, f_n$ is a regular sequence,  then
\[
\reg H^n_Q(R)_j= \reg S_0/c(I)^{-n-j+1}=-dj-n.
\]
\end{Proposition}
\begin{proof}
We first assume that $f_i=x_i$ for $i=1, \ldots, n$  and set $R'=S/gS$ where $g=\sum_{i=1}^nx_iy_i $. By the statement after \cite[Proposition 4.5]{AR1},  $H^n_Q(R')_j$  has the following free $S_0$-resolution
\begin{eqnarray*}
 0\to S_0^{\beta_n}(j)\to \dots \to S_0^{\beta_3}(n+j-3)\to
 S_0^{\beta_2}(n+j-2) \to \\ S_0^{\beta_1}(-1) \to S_0^{\beta_0}\to H^n_Q(R')_j\to 0.
\end{eqnarray*}
The map $\pi:S_0 \to S_0$, $x_i\mapsto f_i$,  is a flat endomorphism, because $f_1,\ldots,f_n$ is a regular sequence. Applying $-\tensor_{S_0}T$ to the above exact sequence, where $T$ is $S_0$ viewed as an $S_0$-module via $\pi$, we get the following $S_0$-free resolution for $H^n_Q(R)_j$
\begin{eqnarray*}
 0\to S_0^{\beta_n}(dj)\to \dots \to S_0^{\beta_3}(d(n+j-3))\to S_0^{\beta_2}(d(n+j-2)) \to\\  S_0^{\beta_1}(-d) \to S_0^{\beta_0}\to H^n_Q(R)_j\to 0.
\end{eqnarray*}
It follows that $\reg H^n_Q(R)_j= -dj-n.$ Now we prove the second equality. As before, we first assume that $f_i=x_i$ for $i=1, \ldots, n$ and hence $c(I)=\mm_0$ where $\mm_0$ is the graded maximal ideal of $S_0$. We set $k=-n-j+1$.  Due to the well-known fact that  $\mm_0^k$ has a linear resolution,  the minimal graded  free $S_0$-resolution of $S_0/\mm_0$ is of the form
\[
0\to S_0^{\beta_n}(n+k-1)\to \dots \to
 S_0^{\beta_1}(k+1) \to \\ S_0^{\beta_0}(k) \to S_0\to S_0/\mm_0^k\to 0.
\]

 Hence $ \reg S_0/\mm_0^k=k-1.$ As before, by using flatness of $\pi$, we obtain the following free resolution for $S_0/c(I)^k$
\[
0\to S_0^{\beta_n}(d(n+k-1))\to \dots \to
 S_0^{\beta_1}(d(k+1)) \to \\ S_0^{\beta_0}(dk) \to S_0\to S_0/c(I)^k\to 0.
\]
Hence $\reg S_0/c(I)^k=d(n+k-1)-n=-dj-1$.
\end{proof}
For the proof of the next corollary we need the following
\begin{Lemma}
\label{regsum}
Let   $R=S/ghS$  and  $R'=S/hS$ where $g\in S_0$ is a homogeneous polynomial and $h\in S$ is a bihomogeneous polynomial. Then
\[
\reg  H^n_Q(R)_j=  \reg H^n_Q(R')_j + \deg g.
\]
\end{Lemma}
\begin{proof}
Let $F/U_j$ be the standard  presentation of  $H^n_Q(R')_j$, as described in  \cite[Section~3, page 322]{AR1}.  Then it follows that $F/gU_j$ is the standard presentation of $H^n_Q(R)_j$. This yields the desired conclusion.
\end{proof}

\begin{Corollary}
\label{twosummands}
Let $R=S/fS$ where $f=f_1y_1+f_2y_2$ is a bihomogeneous polynomial in $S$ of bidegree $(d, 1)$,  and set $g=\gcd(f_1, f_2)$. Then
\[
\reg  H^2_Q(R)_j=\left\{
\begin{array}{cc}
-(d-\deg g)j+\deg g-2 & \text{if  $\deg g<d$},\\
\deg g & \text{if   $\deg g=d$.}
\end{array}
\right.
\]
\end{Corollary}
\begin{proof}
If $\deg g<d$, then we may write $f=gh$ where $h=h_1y_1+h_2y_2$ with $\gcd(h_1, h_2)=1$ and  $\deg h_i>0$ for $i=1, 2$. Note that $h_1, h_2$ is a regular sequence. Hence by Proposition \ref{regular} and Lemma \ref{regsum} we have the first equality. If $\deg g=d$, then $\deg h_1=\deg h_2=0$. Hence $H^2_Q(R')=0$ where $R'=S/hS$. Therefore, the second equality follows from Lemma \ref{regsum}.
\end{proof}

\section{An upper bound for the regularity of local cohomology  of multigraded modules}

Let $K$ be a field and $S=K[x_1,\ldots,x_m,y_1,\ldots,y_n]$ be the polynomial ring over $K$ in the variables $x_1,\ldots,x_m,y_1,\ldots,y_n$. We consider $S$ as standard $\ZZ^m\times \ZZ^n$-graded $K$-algebra. Let $M$ be a finitely generated  $\ZZ^m\times \ZZ^n$-graded $S$-module. Computing  local cohomology by using the \v{C}ech complex shows that $H^s_Q(M)$ is naturally $\ZZ^m\times \ZZ^n$-graded. Therefore, in view of the fact that
$
H^s_Q(M)_j=\Dirsum_{k}H^s_Q(M)_{(k,j)}
$
we see that  the  $\ZZ$-graded components $H^s_Q(M)_j$ of local the cohomology modules  $H^s_Q(M)$ are naturally $\ZZ^m$-graded $S_0$-modules where $S_0$ is the standard $\ZZ^m$-graded  $K$-algebra $K[x_1,\ldots,x_m]$, see \cite[Section 1]{AR4} for more details. In particular,  $H^s_Q(M)_j$ may also be viewed as a $\ZZ$-graded module over $S_0$.
We recall the following theorem from \cite{BH1}.
\begin{Theorem}
\label{BH}
Let $N$ be a finitely generated $\ZZ^m$-graded $S_0$-module. Let

\[
0 \rightarrow F_k \rightarrow \cdots \rightarrow F_1 \rightarrow F_0 \rightarrow N \rightarrow 0,
\]
be  the minimal graded free $S_0$-resolution of $N$ with  $F_i=\Dirsum_ {j=1}^{t_i} S_0(-a_{ij})$ for $i=1,\ldots,k$. Assume that  the first multigraded shifts $a_{0j}$  of $N$ belong to  $\NN^m$. Then for all $i$ and all $j=1, \ldots, t_i$ we have
\[
x^{a_{ij}}| \lcm(x^{a_{11}}, \ldots, x^{a_{1{t_1}}}).
\]
\end{Theorem}

As an immediate consequence we obtain

\begin{Corollary}
\label{onlythisistrue}
The regularity of $N$ is bounded by a constant $c$ which  only depends on the shifts $a_{ij}$ with $i\leq 1$.
\end{Corollary}

As a main result of this section we  have

\begin{Theorem}
\label{againinessen}
Let $M$ be a finitely generated  $\ZZ^m\times \ZZ^n$-graded $S$-module. Then there exists an integer $c$, which only depends on the $x$-shifts of the bigraded resolution of $M$,  such that
\[
|\reg H^s_Q(M)_j|\leq c\quad \text{for all $s$ and all $j$}.
\]
\end{Theorem}
\begin{proof}
By applying a suitable multigraded shift to $M$ we may assume that all generators of $M$ have multidegrees belonging to $\NN^m\times \NN^n$. Then all shifts in the multigraded resolution of $M$ belong to  $\NN^m\times \NN^n$.
Let
\[
\FFF: 0 \rightarrow F_l \overset {\phi_l} \longrightarrow \cdots \longrightarrow F_1 \overset {\phi_1} \longrightarrow F_0 \overset {\phi_0} \longrightarrow M \rightarrow 0,
\]
 be a $\ZZ^m\times \ZZ^n$-graded free resolution of $M$ where $ F_i= \Dirsum_{k=1}^{t_i}S(-a_{ik},-b_{ik})$ for $i=1,\ldots,l$.
Applying the functor $H_{Q}^n(-)_j$ to this resolution yields a graded complex of free $\ZZ^m$-graded modules
\[
H_{Q}^n(\FFF)_j: 0 \rightarrow H_{Q}^n(F_l)_j \overset {\phi_l^*}\longrightarrow \cdots \longrightarrow H_{Q}^n(F_1)_j \overset {\phi_1^*}\longrightarrow H_{Q}^n(F_0 )_j\overset {\phi_0^*} \longrightarrow H_{Q}^n(M)_j \rightarrow 0.
\]
Notice that
\begin{eqnarray}
\label{ayesha}
H_{Q}^n(F_i)_j=\Dirsum_{k=1}^{t_i}\Dirsum_{\left|a\right|=-n-j+|b_{ik}|}S_0(-a_{ik})z^a,
\end{eqnarray}
is a finitely generated free $S_0$-module. It follows that the $\ZZ$-graded modules $H_{Q}^s(M)_j$ are all generated in non-negative degrees. In particular,
$\reg H_{Q}^s(M)_j\geq 0$. Thus it suffice to show that there exists an integer  $c$ such that $\reg H_{Q}^s(M)_j\leq c$.

For each  $i=1, \ldots, l$,  consider the exact sequence
\begin{eqnarray}
\label{ayesha1}
0 \rightarrow \Ker \phi_i^* \rightarrow H_{Q}^n(F_i)_j \overset {\phi_i^*}\longrightarrow H_{Q}^n(F_{i-1})_j \rightarrow N_{i-1, j} \rightarrow 0,
\end{eqnarray}
where $N_{i-1,j}=\Coker  \phi_i^* $. It follows from (\ref{ayesha}) and Corollary~\ref{onlythisistrue} that there exists an integer $c_{i-1}$ such that $\reg N_{i-1, j}\leq c_{i-1}$. Note that the constant $c_{i-1}$ does not depend on $j$.   Hence from (\ref{ayesha1}) one obtains $\reg \Im \phi_i^*= \reg N_{i-1, j}+1$ and $\reg \Ker \phi_i^*=  \reg N_{i-1, j}+2$.
By  \cite[Theorem 1.1]{AR1} we have
\[
H^{s}_Q(M)_j\iso H_{n-s}\big(H_{Q}^n(\FFF)_j\big)\iso \Ker \phi_{n-s}^*/\Im \phi_{n-s+1}^*,
\]
which is an isomorphism of $\ZZ^m$-graded  $S_0$-modules.
Therefore,  the exact sequence $ 0 \rightarrow \Im \phi_{n-s+1}^* \rightarrow \Ker \phi_{n-s}^* \rightarrow  H^{s}_Q(M)_j \rightarrow 0$ yields
\begin{eqnarray*}
\reg  H^{s}_Q(M)_j &\leq& \max\{ \reg \Ker \phi_{n-s}^*, \reg \Im \phi_{n-s+1}^*-1\}\\
                       &=&\max \{ \reg N_{n-s-1, j}+2, \reg N_{n-s, j}\} \\
                     &\leq& \max\{ c_{n-s-1}+2, c_{n-s}\}  \leq c,
\end{eqnarray*}
where $c= \max_i\{c_i\}+2$.
\end{proof}

\begin{Corollary}
Let $I\subseteq S$ be a monomial ideal. Then there exists an integer $c$ such that
\[
|\reg H^i_Q(S/I)_j|\leq c\quad \text{for all $i$ and all $j$}.
\]
\end{Corollary}

For the top local cohomology of $K$-algebras with monomial relations  we have the following more precise statement.

\begin{Proposition}
\label{moreprecise}
Let $I=(u_1v_1,\ldots,u_rv_r)$ be a monomial ideal where the $u_i$ are monomials in $K[x_1,\ldots,x_n]$ and the $v_j$ are monomials in $K[y_1,\ldots,y_n]$. We let $J$ be the monomial ideal $K[x_1,\ldots,x_n]$ generated by $u_1,\ldots,u_r$. Then
\[
H^n_Q(S/I)_j\iso (S_0/J)^{\binom{-j-1}{n-1}}.
\]
In particular,  the regularity of $H^n_Q(S/I)_j$ is constant, namely equal to $\reg S_0/J$, for $j\leq -n$.
\end{Proposition}
\begin{proof}
We set $\deg (u_iv_i)=(a_i, b_i)$ for $i=1, \ldots, r$. Let
 \[
 \cdots \longrightarrow \Dirsum_{i=1}^rS(-a_{i}, -b_{i})  \longrightarrow S \longrightarrow S/I \longrightarrow 0,
 \]
be the free presentation of  $S/I$. Then $H^n_Q(S/I)_j$ has the following $S_0$-presentation
\[
 \cdots \longrightarrow G_1 \overset {\phi} \longrightarrow G_0 \longrightarrow H^n_Q(S/I)_j \longrightarrow 0
 \]
by   free $S_0$-modules, where
\[
G_0= \Dirsum_{\left|c\right|=-n-j}S_0z^c \quad\text{and}\quad G_1=\Dirsum_{i=1}^r \Dirsum_{\left|c\right|=-n-j+b_i}S_0(-a_{i})z^c,
\]
 with $\phi$ as described in the proof Theorem~\ref{Anihilate}: for $G_0$ the basis consists the monomials $z^c$ with  $\left|c\right|=-n-j$ and $G_1$ has a basis consisting of the elements $e_iz^c$ with $\deg e_i=a_i$ and $\left|c\right|=-n-j+b_i$. We have $\phi(e_iz^c)=u_iz^{c-\beta_i}$ if $\beta_i\leq c$ where $v_i=y^{\beta_i}$, and otherwise $0$. It follows from this description that each column of the matrix describing  $\phi$ with respect to this basis has only one non-zero entry and the entries of each of this matrix generates $J$. This implies that $H^n_Q(S/I)_j\iso G_0/JG_0$, as desired.
\end{proof}

\end{document}